\documentclass[11pt]{amsart}

\usepackage{tikz}
\usepackage{amssymb}

\newtheorem{theorem}{Theorem}[section]
\newtheorem{corollary}[theorem]{Corollary}
\newtheorem{lemma}[theorem]{Lemma}
\newtheorem{question}[theorem]{Question}
\newtheorem{proposition}[theorem]{Proposition}

\newtheorem{conjecture}[theorem]{Conjecture}

\newtheorem{definition}[theorem]{Definition}

\newtheorem{maintheorem}{Theorem}

\theoremstyle{definition}

\theoremstyle{remark}
\newtheorem{step}{Step}
\newtheorem{case}{Case}

\numberwithin{equation}{section}
\numberwithin{figure}{section}

\newcommand{\Z}{\mathbb Z}
\newcommand{\T}{\mathbb T}

\newcommand{\eps}{\varepsilon}

\newcommand{\U}{\mathcal U}

\newcommand{\diff}{\operatorname{Diff}}
\newcommand{\phc}{\operatorname{Phc}}

\newcommand{\transv}{\pitchfork}
\newcommand{\zeroeq}{\stackrel{\circ}{=}}
\newcommand{\ess}[1]{\overline{#1}^{ess}}

\DeclareMathOperator{\per}{per}
\DeclareMathOperator{\nuh}{Nuh}
\DeclareMathOperator{\pper}{Per}

\newcommand{\eval}[2][\right]{\relax
  \ifx#1\right\relax \left.\fi#2#1\rvert}

\begin{document}
\title[Stable minimality]{Stable minimality of expanding foliations}

\author[Gabriel N\'u\~nez]{Gabriel N\'u\~nez}
\address{Gabriel N\'u\~nez, 1. IMERL, Facultad de Ingenier\'ia,
Universidad de la Rep\'ublica,
Julio Herrera y Reissig 565,
11.300 Montevideo, Uruguay.}
\address{2. Departamento de matem\'atica, 
Facultad de Ingenier\'ia y Tecnolog\'ias,
Universidad Cat\'olica del Uruguay, 
Comandante Braga 2715, 
11.600 Montevideo, Uruguay
}
\email{fnunez@fing.edu.uy}
\author[Jana Rodriguez Hertz]{Jana Rodriguez Hertz}
\address{Jana Rodriguez Hertz, 1.
Department of Mathematics, 
Southern University of Science and Technology of China.
No 1088, xueyuan Rd., Xili, Nanshan District, Shenzhen,Guangdong, China 518055.}
\address{2. SUSTech International Center for Mathematics}
\email{rhertz@sustech.edu.cn}

\begin{thanks}{GN  was supported by Agencia Nacional de Investigaci\'on e Innovaci\'on. The research that gives rise to the results presented in this publication received funds from the Agencia Nacional de Investigaci\'on e Innovaci\'on under the code POS\_NAC\_2014\_1\_102348}
\end{thanks}
\begin{thanks}
{JRH was supported by NSFC 11871262 and NSFC 11871394} 
\end{thanks}
\keywords{ Minimal foliation, stable minimality, stable ergodicity}

\begin{abstract}
We prove that generically in $\diff^{1}_{m}(M)$, if an expanding $f$-invariant foliation $W$ of dimension $u$ is minimal and there is a periodic point of unstable index $u$, the foliation is stably minimal. By this we mean there is a $C^{1}$-neighborhood $\U$ of $f$ such that for all $C^{2}$-diffeomorphisms $g\in \U$, the $g$-invariant continuation of $W$ is minimal. In particular, all such $g$ are topologically mixing. Moreover, all such $g$ have a hyperbolic ergodic component of the volume measure $m$ which is essentially dense. This component is, in fact, Bernoulli.\par
We provide new examples of stably minimal diffeomorphisms which are not partially hyperbolic. 
\end{abstract}
\maketitle

\section{Introduction}
In this paper, we look for mechanisms activating the {\em stable minimality} of an expanding invariant foliation. From now on, let $M$ be a closed Riemannian manifold, and let $f$ be a $C^{1}$-diffeomorphisms in $M$ preserving a smooth volume $m$. An $f$-invariant foliation is {\em expanding} if it is tangent to a $Df$-invariant sub-bundle $E$ of the tangent bundle $TM$ such that $\|Df(x)v\|>1$ for all unit vectors $v\in E_{x}$, for every $x\in M$. A foliation is {\em minimal} if every leaf of the foliation is dense. \par
An $f$-invariant foliation $W$ is {\em stably minimal} if there exists a $C^{1}$-neighbor\-hood $\U(f)$ of $f$ in $\diff^{1}_{m}(M)$ such that 
\begin{enumerate}
 \item For each $g\in\U$ there exists a $g$-invariant foliation $W_{g}$ such that the fiber bundle $TW_{g}$ varies continuously for $g\in\U(f)$, where $W_{f}=W$
 \item $W_{g}$ is minimal for all $g\in \U(f)\cap\diff^{2}_{m}(M)$
\end{enumerate}
With this definition, a stably minimal $f$-invariant foliation could be not minimal. However, if $f\in\diff^{2}_{m}(M)$, every stably minimal $f$-invariant foliation is minimal. Note that minimality of an invariant foliation is a $G_{\delta}$-property under condition (1) above; hence, the generic stably minimal $f$-invariant foliation will be minimal, even if $f$ is only $C^{1}$.\par
We obtain the following result:

\begin{maintheorem}\label{teo.stable.minimality} For a generic $f\in \diff^{1}_{m}(M)$, if $W$ is a minimal expanding $f$-invariant foliation, and there exists a hyperbolic periodic point $p$ with unstable index $u(p)=\dim W$, then $W$ is stably minimal. In particular, all $C^{2}$-volume preserving diffeomorphisms in a $C^{1}$-neighborhood of $f$ are topologically mixing.
\end{maintheorem}

The hypothesis of the existence of a hyperbolic periodic point with this property may strike as awkward. However, without it we could have a problem as the following, which remains open:
\begin{question} Is the strongest foliation of an Anosov foliation always minimal? 
In other words, let $M$ a closed manifold with $\dim M\geq 3$.  Let us assume that the tangent bundle splits into 3 $Df$-invariant sub-bundles $TM=E^{uu}\oplus E^{u}\oplus E^{s}$, so that for each $v^{\sigma}\in E^{\sigma}$ unit vectors $\sigma=uu,u,s$, we have
$$\|Df(x)v^{s}\|<1<\|Df(x)v^{u}\|<\|Df(x)v^{uu}\|.$$
Then, there exists an $f$-invariant foliation $W^{uu}$ pointwise tangent to $E^{uu}$. Is $W^{uu}$ always minimal? You may add the hypothesis $f\in C^{2}$ if necessary. This is not even known in the case where $\dim E^{u}=1$, or even when $\dim M=3$. Of course, in case $f$ is linear, the answer is always positive. 
\end{question}

We would like to mention some related results. In \cite{BDU2002} it is proven that for 3-dimensional manifolds there is an open and dense subset of robustly transitive diffeomorphisms (that is, diffeomorphisms in the $C^{1}$-interior of transitive diffeomorphisms) far away from tangencies so that either the unstable or the stable foliation is robustly minimal. By robustly minimal it is meant that they are minimal in a $C^{1}$-open neighborhood. This result was later generalized in \cite{HHU2007} for robustly transitive partially hyperbolic diffeomorphisms with 1-dimensional center bundle. \par
Another related result is \cite{PuSa2006}. There it is proved the robust minimality of the stable foliation of a partially hyperbolic diffeomorphism $f$ under the following conditions (1) $W^{s}_{f}$ is minimal (2) $f$ satisfies the SH property. The SH property requires that for any unit disc in any unstable leaf $W^{u}_{f}(x)$ there is a point $y$ where the central bundle $E^{c}$ has a uniform expanding behavior along the future orbit of $y$. \par
These three results are stronger in the sense that they hold in a whole $C^{1}$-open set and not just in the intersection of a $C^{1}$-open set with $\diff^{2}_{m}(M)$. On the other hand, all these three results require partial hyperbolicity. In this sense our Theorem \ref{teo.stable.minimality} is stronger in that it only requires generically the presence of a minimal expanding foliation and a hyperbolic periodic point of an adequate index. No partial hyperbolicity is required. See Section \ref{section.final} for a mechanism to obtain non partially hyperbolic stably minimal expanding foliations. We stress that nevertheless, the existence of a {\em dominated splitting} will follow from the hypothesis, see the beginning of Section \ref{section.proof.thm.b}. A diffeomorphism $f$ has a {\em dominated splitting} if the tangent bundle over $M$ splits into two $Df$-invariant subbundles $TM= E \oplus F$ such that given any $x\in M$, any unitary vectors $v_E \in E(x)$ and $v_F \in F(x)$:
$$\parallel Df^{N}(x) v_E \parallel \leq \frac{1}{2}\parallel Df^{N}(x) v_F \parallel$$
for some $N>0$ independent of $x$.\newline\par 
Recently, minimality has been proven a generic mechanism to activate not only robust topologically mixing properties but also stable ergodicity and even stable Bernoulliness for $3$-dimensional manifolds \cite{NH19}. Even though the statement of this result is deeply connected to the results in this paper, the techniques used there are completely different. \par 
{\em Ergodicity} is a frequent assumption in physical modeling. A diffeomorphism $f$ is {\em ergodic} if it has the same behavior averaged over time as averaged over the space of all states, or, equivalently, if every measurable $f$-invariant set has either full or null measure. Since one often deals with perturbations of $f$, it is of particular interest to study the mechanisms that activate {\em stable ergodicity}. A volume preserving diffeomorphism $f$ is {\em stably ergodic} if there exists a $C^{1}$-neighborhood ${\mathcal U}\subset \diff^{1}_{m}(M)$ such that all $C^{2}$ diffeomorphisms $g\in {\mathcal U}$ are ergodic. The requirement that the surrounding diffeomorphisms $g$ be $C^{2}$ is due to the following open question:

\begin{question}
 Does there exist a $C^{1}$-stably ergodic diffeomorphism? In other words, is there a $C^{1}$-open set of volume preserving ergodic diffeomorphisms?
\end{question}

With the above definition, a $C^{1}$-stably ergodic diffeomorphism might not be ergodic. Even though it is somewhat awkward, we will keep this notation for practical reasons. Observe that, since ergodicity is a $G_{\delta}$-property, the $C^{1}$-generic stably ergodic diffeomorphism is indeed ergodic.\newline\par

The first known mechanism to activate stable ergodicity is {\em hyperbolicity} \cite{AS1967}. A diffeomorphism $f$ is {\em hyperbolic} or {\em Anosov} if there is a $Df$-invariant splitting of the tangent bundle $TM=E^{s}\oplus E^{u}$ such that, for a suitable Riemannian metric, all unit  vectors $v^{s}\in E^{s}_{x}$ and $v^{u}\in E^{u}_{x}$ satisfy:
$$\|Df(x)v^{s}\|<1<\|Df(x)v^{u}\|.$$
An Anosov diffeomorphism has always a dominated splitting. \par
In 1995, Pugh and Shub conjectured that ``a little hyperbolicity goes a long way toward guaranteeing stable ergodicity''. What they had in mind in that moment was {\em partial hyperbolicity}. A diffeomorphism $f$ is {\em partially hyperbolic} if there is a $Df$-invariant splitting of the tangent bundle $TM=E^{s}\oplus E^{c}\oplus E^{u}$ such that, for a suitable Riemannian metric, all unit vectors $v^{\sigma}\in E^{\sigma}_{x}$ with $\sigma=s,c,u$ satisfy
$$\|Df(x)v^{s}\|<\|Df(x)v^{c}\|<\|Df(x)v^{u}\|$$
$$\|Df(x)v^{s}\|<1<\|Df(x)v^{u}\|.$$
An Anosov diffeomorphism is partially hyperbolic, with $E^{c}=\{0\}$. A partially hyperbolic diffeomorphism has a dominated splitting. Partial hyperbolicity has been recently shown to be a generic mechanism activating stable ergodicity \cite{ACW2017}. \newline\par
Following Pugh and Shub, we would like to propose ``a little hyperbolicity'' as a generic mechanism activating stable ergodicity. How far can we go in asking just a little? In 2012, the second author proposed the first author the following problem:

\begin{conjecture}\label{conjecture.dichotomy} \cite{NH19}
 Generically in $\diff^{1}_{m}(M)$, if $f$ has positive metric entropy with respect to Lebesgue measure, then $f$ is stably ergodic. 
\end{conjecture} 

The problem was stated in that moment in dimension 3, because Theorem \ref{teo.jana.acw} was then only known to hold in dimension less or equal than 3.  After it was proven to hold in any dimension, it is natural to extend the conjecture to any dimension. This is as little hyperbolicity as one can get: generically, positive metric entropy activates stable ergodicity. From Theorem \ref{teo.jana.acw} it follows that in that case generically there is also a dominated splitting and {\em non-uniform hyperbolicity}: all Lyapunov exponents are non-zero almost everywhere. \newline\par 

Conjecture \ref{conjecture.dichotomy} seems far to be solved with the current techniques. However, the following could be an approach in dimension 3:
\begin{conjecture}\label{conjecture.minimal}
 Generically in $\diff^{1}_{m}(M^{3})$ if $f$ has positive metric entropy, then there exists a minimal invariant expanding or contracting foliation. 
\end{conjecture}

From Theorem \ref{teo.jana.acw} it follows that generically in dimension 3, the fact that $f$ has positive metric entropy implies that there is a dominated splitting. One of the subbundles of the splitting is one-dimensional. Domination then implies its hyperbolicity, namely, that it is either contracting or expanding (this is left as an exercise to the reader). Hyperbolicity of this bundle implies it is integrated to an expanding or contracting foliation. We conjecture that at least one of this expanding or contracting foliations is minimal. Of course Conjecture \ref{conjecture.dichotomy} could follow without the validity of Conjecture \ref{conjecture.minimal}. \par
We want to state here a more modest conjecture proposing minimality of an invariant expanding or contracting foliation as a generic mechanism activating stable ergodicity. Namely, that Theorem A in \cite{NH19} holds in any dimension:

\begin{conjecture}\label{conjecture.minimality.2}
 Generically in $\diff^{1}_{m}(M)$, the existence of a minimal invariant expanding or contracting foliation implies stable ergodicity, and even {\em stable Bernoullines}.
\end{conjecture}

A diffeomorphism $f\in\diff^{1}_{m}(M)$ is {\em stably Bernoulli} if there exists a $C^{1}$-neighborhood $\U\subset\diff^{1}_{m}(M)$ of $f$ such that all $g\in\U\cap \diff^{2}_{m}(M)$ are Bernoulli, that is, are metrically isomorphic to a Bernoulli shift. \par

As happens with other definitions in this paper, a $C^{1}$ stably Bernoulli diffeomorphism might not be Bernoulli. We are aware that this is not a standard definition, but for practical reasons we will state it like this. Of course, a $C^{2}$ stably Bernoulli diffeomorphism is Bernoulli. Bernoulliness is not necessarily a $G_{\delta}$-property, whence we cannot say that the generic stably Bernoulli diffeomorphism is Bernoulli a priori. \par

The following theorem is an approach to proving Conjecture \ref{conjecture.minimality.2}:

\begin{maintheorem}\label{theorem.B}
For a generic diffeomorphism $f\in\diff^{1}_{m}(M)$, if there exists a minimal invariant expanding foliation $W$ for which there is a hyperbolic periodic $p$ with unstable index $u(p)=\dim W$, then:\par
There exists a $C^{1}$-neighborhood ${\mathcal U}(f)\subset\diff^{1}_{m}(M)$ such that $\forall g\in{\mathcal U}(f)\cap \diff^{2}_{m}(M)$
there is a hyperbolic ergodic component $\phc_{g}(q_{g})$ whose essential closure satisfies $$\ess{\phc_{g}(q_{g})}=M$$
\end{maintheorem}

A measurable set $A$ such that $m(A)>0$ is an {\em ergodic component} if $f|A$ is ergodic. The ergodic component is {\em hyperbolic} if all Lyapunov exponents of $f$ on $A$ are different from zero (see definitions in Section \ref{section.definitions}). The {\em essential closure} of a set $A$ is the set
$$\ess{A}=\{x\in M: \forall \eps>0 \quad m(B_{\eps}(x)\cap A)>0\}$$
A set whose essential closure is the whole manifold $M$ is called {\em essentially dense}. 
What is special about Theorem \ref{theorem.B} is that the ergodic component $\phc_{g}(q_{g})$ is given explicitly. It consists of all points satisfying a topological condition. See Section \ref{section.definitions}. \newline\par

The paper will be organized as follows: In Section \ref{section.definitions} the basic definitions and results necessary for the proof will be introduced.  Theorem \ref{theorem.B} is proven in Section \ref{section.proof.thm.b}. Theorem \ref{teo.stable.minimality} is proven in Section \ref{section.thm.A}. In the final Section \ref{section.final}, we describe a mechanism to obtain new examples, and provide them.  
 
\section{Basic concepts}\label{section.definitions}

Let $f\in\diff^{1}_{m}(M)$ be a volume preserving diffeomorphism. We will say that $\lambda(x,v)$ is the {\em Lyapunov exponent}  associated to $v\in T_{x}M$ if
$$\lambda(x,v)=\limsup_{n\to\pm\infty}\frac{1}{n}\log\|Df^{n}(x)v\|$$
For $m$-almost every $x\in M$, there are finitely many Lyapunov exponents $\lambda_{1}(x),\dots,\lambda_{k}(x)$ in $T_{x}M$, and there is a measurable $Df$-invariant splitting, called the {\em Oseledets splitting} 
$$T_{x}M=E^{1}_{x}\oplus\dots\oplus E^{k}_{x}$$
such that $\lambda(x,v_{i})=\lambda_{i}(x)$ for all $v_{i}\in E^{i}_{x}\setminus\{0\}$. See for instance, \cite{Pesin1977}. 
We denote by $\nuh(f)$ the set of $x$ such that all $\lambda(x,v)$ are different from zero. 
The measure $m$ is called {\em hyperbolic} if $\nuh(f)$ has full $m$-measure, that is, if all Lyapunov exponents are different from zero almost everywhere. 
For simplicity, we will denote
\begin{equation}\label{zipped.oseledets.splitting}
T_{x}=E^{-}_{x}\oplus E^{0}_{x}\oplus E^{+}_{x} 
\end{equation}

the splitting such that $\lambda(x,v^{+})>0$ for all $v^{+}\in E^{+}_{x}$, $\lambda(x,v^{0})=0$ for all $v^{0}\in E^{0}_{x}$ and $\lambda(x,v^{-})<0$ for all $v^{-}\in E^{-}_{x}$. We call this splitting the {\em zipped Oseledets splitting}.
\par
We will say that a measurable set $A$ is an {\em ergodic component} of $m$ if $m(A)>0$ and $f|A$ is ergodic. $A$ is a {\em hyperbolic} ergodic component if $A\subset \nuh(f)$.\par
For $x\in M$, let us define
\begin{equation}\label{eq.invariant.manifolds}
 W^{\pm}(x)=\left\{y\in M: \limsup_{n\to\infty}\frac{1}{n}\log d(f^{\mp n}(x),f^{\mp n}(y))<0\right\}
\end{equation}
If $f\in\diff^{2}_{m}(M)$ then for $m$-almost every point $x$, $W^{+}(x)$ and $W^{-}(x)$ are smooth immersed manifolds \cite{Pesin1977}. For $f\in\diff^{1}_{m}(M)$ this is not necessarily true \cite{Pugh1984}.
However, if the zipped Oseledets splitting is dominated, then both $W^{+}(x)$ and $W^{-}(x)$ are immersed manifolds for $m$-almost every point, see \cite{ABC2011}.\par

Following \cite{HHTU11}, given a hyperbolic periodic point $p\in M$ we define the {\em stable Pesin homoclinic class} of $p$ by 
\begin{equation}\label{eq.phc.-}
\phc^{-}(p)=\{x: W^{-}(x)\transv W^{u}(o(p))\ne\emptyset\} 
\end{equation}
where $W^{u}(o(p))$ is the union of the unstable manifolds of $f^{k}(p)$, for all $k=0,\dots,{\rm per}(p)-1$. $\phc^{-}(p)$ is invariant and saturated by $W^{-}$-leaves. See Figure \ref{phc-}. We will also denote by $o(p)$ the orbit of $p$. 
\begin{figure}
 \includegraphics[width=.7\textwidth]{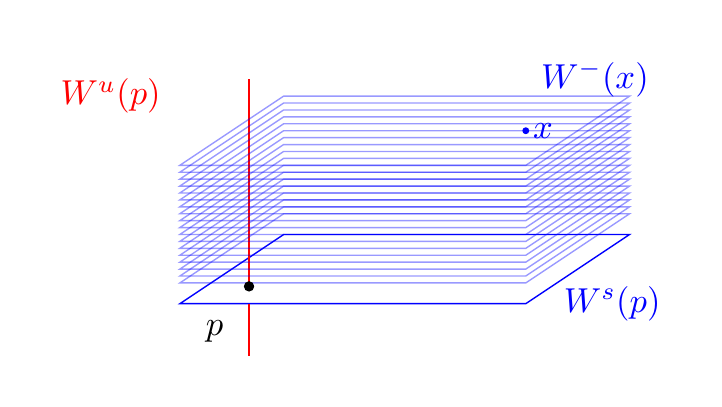}
 \caption{\label{phc-}$\phc^{-}(p)$}
\end{figure}
Analogously, we define 
\begin{equation}\label{eq.phc.+}
\phc^{+}(p)=\{x: W^{+}(x)\transv W^{s}(o(p))\ne\emptyset\} 
\end{equation}
\begin{figure}
 \includegraphics[width=.7\textwidth]{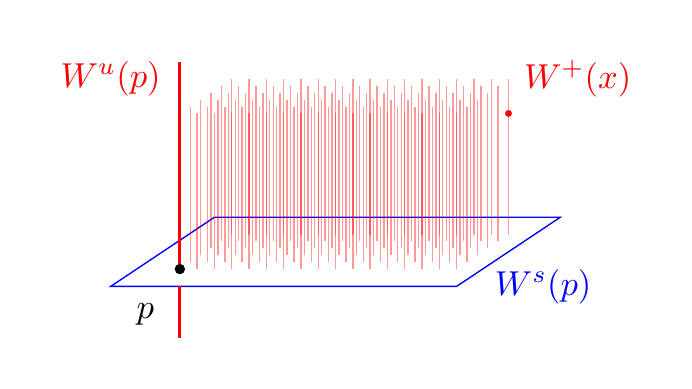}
 \caption{\label{phc+}$\phc^{+}(p)$}
\end{figure}
which is invariant and saturated by $W^{+}$-leaves. See Figure \ref{phc+}. We denote by $\phc(p)$ the intersection of $\phc^{-}(p)$ and $\phc^{+}(p)$.\par 

If there exists an expanding foliation $W^{u}$, we will denote 
\begin{equation}\label{eq.phc.u}
 \phc^{u}(p)=\{x\in M: W^{u}(x)\transv W^{s}(o(p))\ne\emptyset\}
\end{equation}
Analogously we define $\phc^{s}(p)$ if a contracting foliation $W^{s}$ is given. The foliation will be clear from the context, if it is not, we will denote these sets by $\phc^{W}(p)$, where $W$ is given.

\section{Proof of Theorem \ref{theorem.B}}\label{section.proof.thm.b}
This section is devoted to the proof of Theorem \ref{theorem.B}. Let $f\in\diff^{1}_{m}(M)$ be a generic diffeomorphism with an invariant expanding minimal foliation $W_{f}$ and let $p$ be a hyperbolic periodic point with unstable index $u(p)=\dim W_{f}$. This implies that for $m$-almost every $x\in M$, $\lambda(x,v)>0$ for all $v\in T_{x}W_{f}\setminus \{0\}$. The following theorem applies:

\begin{theorem}\cite{mane1984,bochi2002,JRH2012,ACW16}\label{teo.jana.acw} For a generic $f\in\diff^{1}_{m}(M)$, either all Lyapunov exponents are zero $m$-almost everywhere, or else:
\begin{enumerate}
 \item $f$ is ergodic
 \item the Oseledets splitting is dominated. Call the zipped Oseledets splitting $TM=E^{+}\oplus E^{-}$
\item there exists a hyperbolic periodic point $q$ with $u(q)=\dim E^{+}$ such that $\phc(q)\zeroeq\nuh(f)\zeroeq M$
\end{enumerate} 
\end{theorem}
The theorem above implies that generically an expanding invariant foliation has a continuation, so it makes sense to talk about stable minimality of an expanding invariant foliation. Indeed, generically the existence of an expanding invariant foliation $W_{f}$ implies there exists a dominated splitting $TM=TW_{f}\oplus F$ where $Df|_{TW_{f}}$ is expanding. Since dominated splittings vary continuously in the $C^{1}$-topology, it follows there exists a continuation of the subbundle $E(f)=TW_{f}$ in a $C^{1}$-neighborhood of $f$, so that $E(g)$ is $Dg$ invariant and $Dg|_{E(g)}$ is expanding. Since a dominated expanding bundle is always integrable it follows the existence of an invariant expanding foliation for each $g$ in a $C^{1}$-neighborhood of $f$. \par

From Theorem \ref{teo.jana.acw} it follows the existence of a hyperbolic periodic point $q$ such that 
$$\phc(q)\zeroeq M.$$
Since the unstable index of $p$ satisfies $u(p)=\dim W_{f}$, we deduce there are at least $u(p)$ positive Lyapunov exponents. But $u(q)$ is the maximum number of positive Lyapunov exponents for $f$, therefore $u(q)\geq u(p)$.\par

A {\em superblender}, introduced by Moreira and Silva in \cite{MS12}, is an open set obtained from a perturbation of a horseshoe, after which, either the stable or the unstable manifolds (or both) occupy a larger dimension than it previously had. We will here follow the definition of \cite{ACW2017}. \par

Let $\Lambda$ be a {\em horseshoe}, that is, a transitive, locally maximal hyperbolic set that is totally disconnected and not finite. Assume $f$ admits a dominated splitting over $\Lambda$ of the form $T_{\Lambda}M=E^{u}_{1}\oplus\dots\oplus E^{u}_{\ell}\oplus E^{s}$ so that $Df$ is contracting over $E^{s}$ and expanding over $E^{u}_{1}\oplus\dots\oplus E^{u}_{\ell}$. Given a small open ball $B$ that intersects $\Lambda$, a {\em well placed unstable $k$-strip} is any $k$-disc centered at a point in $B$ with radius much bigger than $B$ that is almost tangent to $E^{u}_{1}\oplus\dots\oplus E^{u}_{k}$, with $k=1, \dots, \ell$, where $\ell=\dim(M)-\dim (E^{s})$. We say that two open submanifolds $K,N$ intersect {\em quasi-transversely} at $z\in M$ if $z\in K\cap N$ and $T_{z}K\cap T_{z}N$ does not contain a non-zero vector. 

\begin{definition}[$s$-stable superblender] Let $\Lambda$ be a horseshoe admitting a dominated splitting of the form $T_{\Lambda}M=E^{u}_{1}\oplus\dots\oplus E^{u}_{\ell}\oplus E^{s}$ so that $Df$ is contracting over $E^{s}$ and expanding over $E^{u}_{1}\oplus\dots\oplus E^{u}_{\ell}$. Let $x\in\Lambda$. A small open ball $Bl^{s}_{\Lambda}(x)$ containing $x$ is an {\em $s$-stable superblender} associated to $\Lambda$ if:
\begin{itemize}
 \item For every $k=1,\dots, \ell$, every well placed $k$-strip in $Bl^{s}_{\Lambda}(x)$  quasi-transversely intersects $W^{s}(y)$ at some point $z\in M$, for some $y\in \Lambda$. 
 \item This property is $C^{1}$-robust
\end{itemize}
See Figure \ref{superblender}
\end{definition}

\begin{figure}[h]
\includegraphics[width=.8\textwidth]{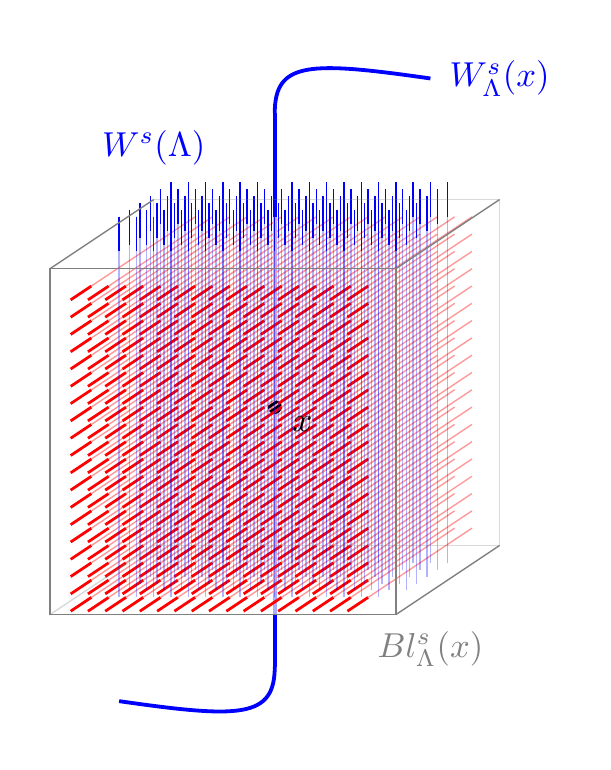}
\caption{\label{superblender} An $s$-stable superblender} 
\end{figure}

The following theorem is related to Theorem \ref{teo.jana.acw} and provides us with stable superblenders in our setting. It follows from \cite[Corollary D]{ACW2017}
\begin{theorem}[Creation of superblenders \cite{ACW2017}\label{thm.superblender}] Generically in $\diff^{1}_{m}(M)$ with positive metric entropy, there exists an $s$-stable superblender associated to a horseshoe $\Lambda$, where $s$ is the stable index of the periodic point $q$ given by Theorem \ref{teo.jana.acw} such that 
$$\phc(q)\zeroeq M.$$

\end{theorem}

Theorem \ref{thm.superblender} implies the existence of a horseshoe $\Lambda$ and an $s$-stable superblender $Bl^{s}_{\Lambda}(x_{0})$ associated to $\Lambda$, where $s$ is the stable index of $q$ and $x_{0}\in\Lambda$. The creation of superblenders strongly uses the volume preserving hypothesis. \par

Let $r$ be any hyperbolic periodic point in $\Lambda$. Then, the following theorem shows that $q$ and $r$ can be considered to be homoclinically related.

\begin{theorem}\cite{AC2012}\label{teo.abdenur.crovisier} Generically in $\diff^{1}_{m}(M)$, all periodic points of the same index are homoclinically related.
\end{theorem}

Call $q_{g},r_{g}$ and $\Lambda_{g}$ the continuations, respectively, of $q,r$ and $\Lambda$. Then there exists a $C^{1}$-neighborhood, which we continue to call $\U$, so that $q_{g}$ and $r_{g}\in\Lambda_{g}$ are homoclinically related for all $g\in\U$. This implies that $q_{g}$ is homoclinically related to all periodic points in $\Lambda_{g}$. \par
Since $W_{f}$ is minimal, all leaves of $W_{f}$ contain a well-placed unstable $k$-strip in $Bl^{s}_{\Lambda}(x_{0})$, where $k=\dim W_{f}=u(p)$. The neighborhood $\U$ can be chosen so that all leaves of $W_{g}$ also contain a well-placed unstable $k$-strip in $Bl^{s}_{\Lambda}(x_{0})$ for every $g\in\U$. Since the superblender property holds, each of these $W_{g}$-leaves will quasi-transversely intersect the $W^{s}_{g}$-leaf of a point in $\Lambda_{g}$. \par
We will need the following lemmas
\begin{lemma}\label{lemma.phc}
If $g\in\diff^{2}_{m}(M)$, then 
$$\overline{W^{s}(q_{g})}=\ess{\phc_{g}(q_{g})}\qquad\text{and}\qquad \overline{W^{u}(q_{g})}=\ess{\phc_{g}(q_{g})}$$ 
\end{lemma}
\begin{proof}
 The inclusion 
 $\overline{W^{s}(q_{g})}\subset\ess{\phc_{g}(q_{g})}$ can be found, for instance, in \cite[Lemma 4.2]{AB2012}. The other inclusion follows easily from the fact that 
$$\phc(q_{g})\subset \overline{W^{s}(q_{g})}$$
which in turn follows from the $\lambda$-lemma.
\end{proof}

Also from the $\lambda$-lemma, we deduce the following
\begin{lemma} \label{lemma.ws} For every $g\in\U\cap\diff^{2}_{m}(M)$
$$W^{s}(\Lambda_{g})\subset \overline{W^{s}(q_{g})}=\ess{\phc_{g}(q_{g})}$$
\end{lemma}

To finish the proof we will need two more lemmas:

\begin{lemma}[Lemma 6.1 \cite{JRH2012}]\label{lemma.jana} If $g\in\diff^{2}_{m}(M)$ and $K$ is a $g$-invariant set such that $K=\ess{K}$ and $m(K)>0$, then, for every $x\in K$
$$W^{s}(x)\cup W^{u}(x)\subset K$$
\end{lemma}

\begin{lemma} \cite{AB2012} \label{avila.bochi} For a generic $f\in\diff^{1}_{m}(M)$ if $q$ is the hyperbolic periodic point of Theorem \ref{teo.jana.acw}, for every $\eps>0$ there exists a $C^{1}$-neighborhood $\U(f)$ of $f$ such that for all $C^{2}$ diffeomorphisms $g$ in $\U(f)$:
 $$m(\phc_{g}(q_{g}))>1-\eps$$
 where the hyperbolic periodic point $q_{g}$ is the continuation of $q$. 
\end{lemma}
\begin{proof} It follows from Lemma 5.1. in \cite{AB2012} and the fact that generically in $\diff^{1}_{m}(M)$ all points of the same unstable index are homoclinically related \cite{AC2012}.
\end{proof}

The following {\em criterion for ergodicity} is crucial and implies that $\phc_{g}(q_{q})$ is an ergodic component:

\begin{theorem}[Theorem A, \cite{HHTU11} Criterion for ergodicity]\label{criterion_hhtu} Let $f:M \rightarrow M$ be a $C^2$-diffeomorphism over a closed connected Riemannian manifold $M$, let $m$ be a smooth invariant measure and $p \in \pper_H(f)$. If $m(\phc^{+}(p))>0$ and $m(\phc^{-}(p))>0$, then
\begin{enumerate}
	\item $\phc^{+}(p) \circeq \phc^{-}(p) \circeq \phc(p)$.
	\item $m|\phc(p)$ is ergodic.
	\item $\phc(p) \subset \nuh(f)$.
\end{enumerate}
\end{theorem}
The notation $A\zeroeq B$ means $m(A\triangle B)=0$.

Now we are in conditions to finish the proof. Let $x\in M$ then, as stated above, $W_{g}(x)$ contains a well placed unstable $k$-strip in $Bl^{s}_{\Lambda}(x_{0})$, where $k=u(p)$, the unstable index of $p$. Since the superblender property holds, this $k$-strip (and hence $W_{g}(x)$) quasi-transversely intersects the $W^{s}_{g}$-leaf of a point in $\Lambda_{g}$. 
Therefore, by Lemma \ref{lemma.ws} $W_{g}(x)$ intersects $\ess{\phc_{g}(q_{g})}$ at a point $y$. But by Lemmas \ref{lemma.jana} and \ref{avila.bochi}, $W_{g}(x)=W_{g}(y)\subset W^{u}(y)$ is contained in $\ess{\phc_{g}(q_{g})}$; hence, $x\in \ess{\phc_{g}(q_{g})}$. This finishes the proof of Theorem \ref{theorem.B}.

\section{Proof of Theorem \ref{teo.stable.minimality}} \label{section.thm.A}

We will use the following criterion of minimality. Some of the ideas of this proposition were already present in \cite{BDU2002}. 

\begin{proposition}[Minimality Criterion] \label{teo.minimality} Given a diffeomorphism $f\in\diff^{1}_{m}(M)$, an expanding $f$-invariant foliation $W$, and a hyperbolic periodic point $p\in\pper(f)$ such that 
\begin{enumerate}
\item the unstable index of $p$, $u(p)$ equals $\dim W$ 
\item $\phc^{W}(p)=M$
\item $\overline{W(p)}=M$ 
\end{enumerate}
 Then $W$ is a minimal foliation.
\end{proposition}

\begin{proof}
\begin{step} The unstable manifold of each point in $x$ not only intersects $W^{s}(o(p))$, but $W^{s}(p)$ itself, that is:
$$\phc^{W}(p)=\{x:W(x)\transv W^{s}(p)\ne\emptyset\}=M$$ 
By invariance of transversality, and since $\phc^{W}(p)=M$, the foliation $W$ is transverse to $W^{s}(f^{k}(p))$ for all $k=0,\dots,\per(p)-1$. But since $W(p)$ is dense in $M$, then $W(f^{k}(p))$ is dense for all $k=0,\dots,\per(p)-1$. This implies that $W(f^{k}(p))\transv W^{s}(p)\ne\emptyset$ for all $k$. Now, for any $x\in M$ there exists a $k=0,\dots,\per(p)-1$ such that $W(x)\transv W^{s}(f^{k}(p))\ne\emptyset$. Since $W(f^{k}(p))\transv W^{s}(p)\ne\emptyset$, the $\lambda$-lemma implies that for some iterate $m$ which is multiple of $\per(f)$, we have $W(f^{m}(x))\transv W^{s}(p)\ne\emptyset$. Applying $f^{-m}$ to the previous intersection, and since $W^{s}(p)$ is $f^{m}$-invariant, we get that $W(x)\transv W^{s}(p)\ne\emptyset$.
\end{step}
\begin{step}
 There exists $K>0$ such that for all $x$, $W(x)\transv W^{s}_{K}(p)\ne\emptyset$.\newline
Here we call $W^{s}_{K}(p)$ the set of points that can be joined to $p$ inside $W^{s}(p)$ by an arc of length less than $K$, for each $K>0$. Indeed, let 
$$\Lambda_{K}=\{x: W(x)\cap W^{s}_{K}(p)=\emptyset\}$$
Then $\Lambda_{K}$ is a compact $W$-saturated set. And $\Lambda_{K+1}\subset \Lambda_{K}$ for all $K\in \Z^{+}$. Call $\Lambda=\bigcap_{K\in\Z^{+}}\Lambda_{K}$. If $\Lambda_{K}$ were non-empty for all $K>0$ then $\Lambda\ne\emptyset$. But this is a contradiction, since $\Lambda\subset M\setminus \phc^{W}(p)$.
\end{step}
\begin{step} For each $\eps>0$ for each $x\in M$ $W(x)\transv W^{s}_{\eps}(p)\ne\emptyset$,\newline
This is just applying $f$ to the previous step $k\per(p)$ times.  
\end{step} 
\begin{step}
 For every $\eps>0$ and every $x$, $W(x)$ is $\eps$-dense.\newline
 Let $K>0$ be such that $W_{K}(p)$ be $\frac{\eps}{2}$-dense. Let $\delta>0$ be such that if $d(x,y)<\delta$ then $d_{H}(W_{K}(x),W_{K}(y))<\frac{\eps}{2}$, where $d_{H}$ is the Hausdorff distance. Now, by previous step, $W(x)\transv W^{s}_{\delta}(p)\ne\emptyset$, so there is $y\in W(x)$ such that $d(y,p)<\delta$. This implies that $d_{H}(W_{K}(y), W_{K}(p))<\frac{\eps}{2}$. Therefore, $W(x)\supset W_{K}(y)$ is $\eps$-dense. 
 \end{step}
 Since this holds for all $\eps>0$, $W$ is minimal.
\end{proof}

The following lemma will be also used together with the Minimality Criterion. It follows straightforwardly from \cite[Lemma 3.2]{NH19}
\begin{lemma}\label{lemma.phc.persists}
If $W$ is an $f$-invariant expanding foliation and $p$ is a hyperbolic periodic point with unstable index $u(p)=\dim W$ such that $\phc^{W}(p)=M$, then there exists a $C^{1}$-neighborhood $\U\subset\diff^{1}_{m}(M)$ such that for all $g\in\U$, 
$$\phc^{W_{g}}(p_{g})=M$$
where $W_{g}$ is the continuation of the invariant expanding foliation $W$ and $p_{g}$ is the continuation of the hyperbolic periodic point $p$.
\end{lemma}

Let $f\in\diff^{1}_{m}(M)$ be a generic diffeomorphism with an invariant expanding minimal foliation $W$ and a hyperbolic periodic point $p$ such that its unstable index satisfies $u(p)=\dim W$. Let $q$ be a hyperbolic periodic point as in Theorem \ref{teo.jana.acw}.  We always have $u(p)\leq u(q)$, where $u(p)$ and $u(q)$ are the unstable indices of $p$ and $q$, respectively. This is because the tangent bundle of $W$ satisfies $TW\subset E^{+}$, where $E^{+}$ is the subbundle of the zipped Oseledets splitting corresponding to the positive Lyapunov exponents. 
\setcounter{case}{0}
\begin{case} $u(p)=u(q)$.\\
In this case, by Theorem \ref{teo.abdenur.crovisier}, $p$ and $q$ are homoclinically related. This persists in a $C^{1}$-neighborhood $\U$ of $f$. The $\lambda$-lemma then implies that
$\ess{\phc_{g}(p_{g})}=\ess{\phc_{g}(q_{g})}$ for all $g\in \U\cap\diff^{2}_{m}(M)$. \par
Theorem \ref{theorem.B} implies that $\ess{\phc_{g}(p_{g})}=M$ for every $g\in\diff^{2}_{m}(M)\cap \U$, and Lemma \ref{lemma.phc} implies that $\overline{W_{g}(p)}=M$. Then the Minimality Criterion (Proposition \ref{teo.minimality}) applies and from Lemma \ref{lemma.phc.persists} we get that $W_{g}$ is minimal for every $g\in\diff^{2}_{m}(M)\cap \U$.
\end{case}
\begin{case} $u(p)<u(q)$. \\
This case is more sophisticated, and it will need the technology of {\em blenders}. We warn the reader that there are many definitions of blenders in the literature. In \cite{BDV}, Chapter 6, there is a complete presentation on the different ways of defining these objects. We will use an approach similar to the one appearing in \cite{HHTU2010}. \par

Let $p$ be a hyperbolic periodic point of period $n$ such that $Df^{n}(p)$ admits the following invariant splitting in the tangent bundle: $T_{p}M=E^{u}_{p}\oplus E^{c}_{p}\oplus E^{s}_{p}$ such that $\dim E^{c}_{p}=1$, $Df^{n}(p)$ is contracting on $E^{c}_{p}\oplus E^{s}_{p}$ and expanding on $E^{u}_{p}$. Let $B$ be a ball near $p$, but not necessarily containing $p$, so that the splitting $T_{p}M=E^{u}_{p}\oplus E^{c}_{p}\oplus E^{s}_{p}$ has a natural continuation in its tangent bundle. Let $s=\dim E^{s}_{p}$. 
A {\em well-placed} $s$-disc $D^{s}$ is an $s$-dimensional disc centered at a point in $B$ with radius much bigger than the radius of $B$ that is almost tangent to $E^{s}$, i.e. the vectors tangent to $D^{s}$ are $C^1$-close to $E^{s}$. A {\em well-placed $(s+1)$-strip} is any $(s+1)$-disc centered at a point in $B$ containing a well placed $s$-disc that is almost tangent to $E^c \oplus E^s$. 
\begin{definition}[$u$-blender] Let $p$ and $r$ be hyperbolic periodic points such that their unstable indices satisfy $u(r)=u(p)+1$. $Bl^{u}(p)$ is a {\em $u$-blender near $p$ activated by $r$} if $Bl^{u}(p)$ is a ball near $p$ not necessarily containing $p$ so that:
\begin{itemize}
 \item every well placed $(s+1)$-strip in $Bl^{u}(p)$ transversely intersects $W^{u}(p)$
 \item $W^{s}(q)$ contains a well placed $s$-disc in $Bl^{u}(p)$
  \item these properties are $C^{1}$-robust
\end{itemize}
Let $p'$ be a hyperbolic periodic point with the same unstable index as $p$. A small open ball $Bl^{u}(p')$ is a {\em $u$-blender associated with $p'$ activated by $r$} if $Bl^{u}(p')=Bl^{u}(p)$, where $Bl^{u}(p)$ is a $u$-blender near $p$ activated by $r$ and $p$ is homoclinically related to $p'$. See Figure \ref{blender}
\end{definition}

\begin{figure}[h]
\includegraphics[width=1\textwidth]{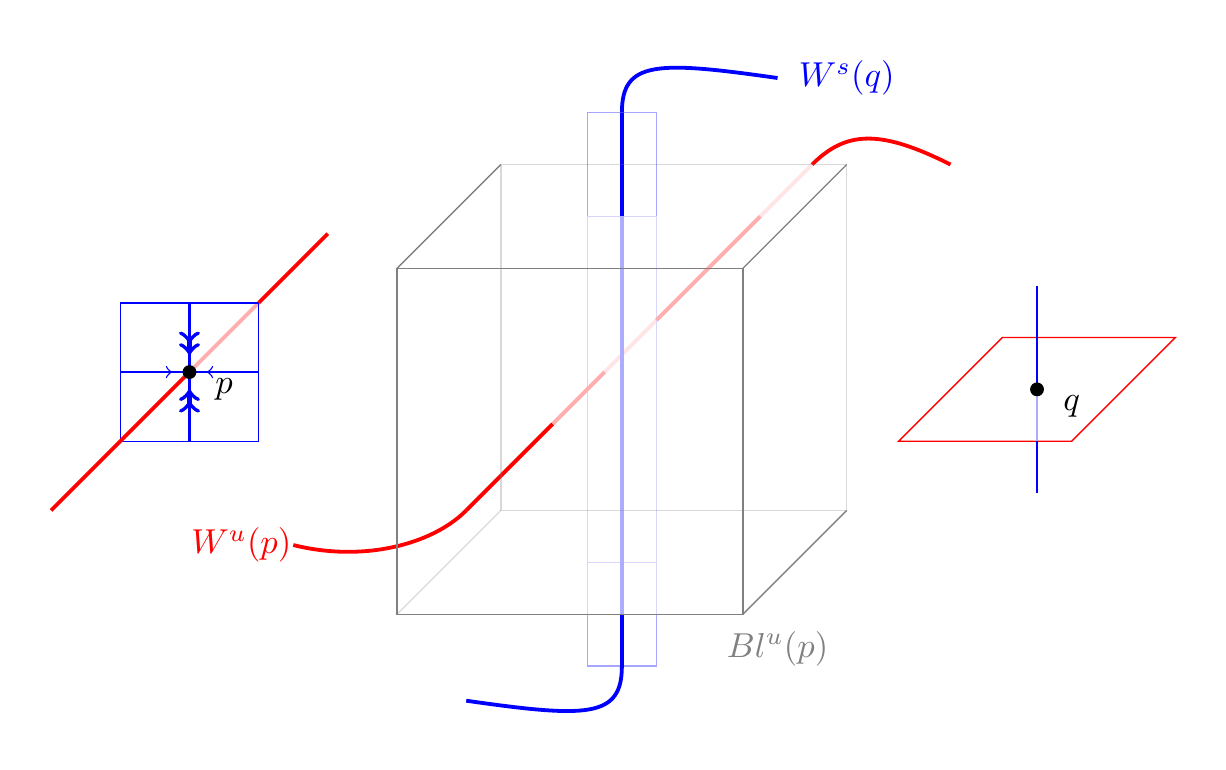}
\caption{\label{blender}A $u$-blender associated to $p$ activated by $r$} 
\end{figure}

We can analogously define {\em $s$-blenders}.\newline\par
Next theorem allows us to obtain $g\in\diff^{1}_{m}(M)$ admitting blenders near some $f\in \diff^{1}_{m}(M)$ with a pair of hyperbolic periodic points with co-index one. It can be found in \cite[Theorem 1.1]{HHTU2010}

\begin{theorem}[Creation of blenders \cite{HHTU2010}] \label{creation.blender}  Let $f \in \diff_m^k(M)$ be such that $f$ has two hyperbolic periodic points $r$ and $p$ of unstable indices $(u+1)$ and $u$ respectly. Then there is $g\in\diff^{k}_{m}(M)$ arbitrarily $C^1$-close to $f$ which admits a $u$-blender associated to the continuation of $p$ activated by the continuation of $r$.\par
In particular, generically in $\diff^{1}_{m}(M)$, for every pair of hyperbolic periodic points $p$ and $r$ with unstable indices $u$ and $(u+1)$ respectively, there exists a $u$-blender $Bl^{u}(p)$ associated to $p$ activated by $r$.
\end{theorem}

The following lemma states the main property of $u$-blenders; namely, that the unstable manifold of $p$ ``occupies'' as much space as $W^{u}(r)$. The proof essentially follows \cite[Lemma 6.12]{BDV}, but since their definition of blender is not exactly ours, we include the proof here for completeness.

\begin{lemma}\label{propiedad.blender}
Let $p$ and $r$ be hyperbolic periodic points with unstable indices $u$ and $(u+1)$ respectively, so that $r$ activates a $u$-blender associated to $p$. Then
$$W^{u}(r_{g})\subset \overline{W^{u}(p_{g})}$$
where $p_{g}$ and $r_{g}$ are the continuations of $p$ and $r$, respectively, in a $C^{1}$-neighborhood of $f$.
\end{lemma}
\begin{proof}
Without loss of generality, we may assume that $Bl^{u}(p)$ is a blender near $p$ associated with $r$. Let $x\in W^{u}(r)$ and consider any neighborhood $U$ of $x$. Take $N=\per(p)\per(r)$. Let $D^{s}$ be the well placed $s$-disc contained in $W^{s}(r)$ as in the definition of $u$-blenders. Due to $\lambda$-lemma, for large $k$, $f^{-kN}(U)$ contains discs as $C^{1}$-close to the disc $D^{s}\subset W^{s}(r)$ as we wish. Since $f^{-kN}(U)$ is open, it contains a well placed $(s+1)$-strip. By the definition of blender $W^{u}(p)$ cuts this strip, and hence $f^{-kN}(U)$. Since $W^{u}(p)$ is $f^{N}$-invariant, $W^{u}(p)\cap U\ne \emptyset$, so the claim is proved. 
\end{proof}

We are also going to apply the Minimality Criterion (Proposition \ref{teo.minimality}) in this case. Since $\phc^{W}(p)=M$, by Lemma \ref{lemma.phc.persists} $\phc^{W_{g}}(p_{g})=M$ for all $g$ in a $C^{1}$-neighborhood $\U$ of $f$, where $p_{g}$ is the continuation of $p$.\par
Remember the following:
\begin{theorem}  \label{teo.jiagang}
For a generic diffeomorphism $f\in \diff^{1}_{m}(M)$, if there are hyperbolic periodic points $p$ and $q$ of indices $u$ and $(u+c_{1})$, respectively, then there is a dense set of hyperbolic points of index $(u+i)$ for each $i=0,\dots, c_{1}$.
\end{theorem}
(The following  theorem was originally stated, for generic diffeomorphisms in $\diff^{1}(M)$ in \cite[Theorem A]{BDPR2000}. For the volume preserving case it can be found in \cite[Lemma 3.9]{LSY2012})
By Theorem \ref{teo.jiagang}, there exists a finite sequence of hyperbolic periodic points $p=q_{0},q_{1},\dots,q_{\ell}=q$, such that 
their unstable indices satisfy $u(q_{i+1})=u(q_{i})+1$, for $i=0,\dots,\ell-1$, where $\ell=u(q)$. Theorem \ref{creation.blender} implies that generically there exist $u$-blenders $Bl^{u}(q_{i})$ associated to $q_{i}$ and activated by $q_{i+1}$, for $i=0,\dots,\ell-1$. Lemma \ref{propiedad.blender} then implies:
\begin{equation}\label{eq.cascada.blenders}
 W^{u}(q_{g})\subset \overline{W^{u}(q^{g}_{\ell-1})}\subset\dots\subset\overline{W^{u}(q^{g}_{1})}\subset \overline{W^{u}(p_{g})}
\end{equation}
for all $g$ in a $C^{1}$-neighborhood $\U$ of $f$, where $p_{g}, q_{g}$ and $q^{g}_{i}$ are the continuations, respectively, of $p,q$ and $q_{i}$, for $i=1,\dots,\ell-1$.\par
Theorem \ref{theorem.B} and Lemma \ref{lemma.phc} imply that $W^{u}(q_{g})$ is dense in $M$. The hypothesis of volume preserving is crucial in this step, since it is not known, to the best of our knowledge, how to obtain that $\overline{W^{u}(q_{g})}=M$ robustly in a $C^{1}$-neighborhood of $f$ without it. Lemma \ref{propiedad.blender} now implies that $\overline{W^{u}(p_{g})}=M$.
The Minimality Criterion now applies and we obtain the stable minimality of $W$.
\end{case}
This finishes the proof of Theorem \ref{teo.stable.minimality}.

\section{New examples}\label{section.final}
As we mentioned in the introduction, the known mechanisms to obtain stable minimality require partial hyperbolicity. Theorem \ref{teo.stable.minimality} requires the existence of a generic minimal expanding foliation together with a hyperbolic periodic point with the same unstable index as the dimension of the foliation. However, a generic minimal expanding foliation is not always easy to obtain. The following proposition facilitates a way to obtain generic minimality of the continuation of a foliation in a $C^{1}$ open set. 
\begin{proposition}\label{proposition.criterion} Let $W$ be an expanding invariant foliation for a diffeomorphism $f\in\diff^{1}_{m}(M)$ and $E$ an invariant sub-bundle of $TM$ such that the splitting $TM=TW\oplus E$ is dominated. Let $p$ be a hyperbolic periodic point such that the unstable index $u(p)$ satisfies $u(p)=\dim W$. Let $U$ be an open set such that the $W$-leaf of each point in $U$ transversely intersects the local stable manifold of $p$. If every $W$-leaf intersects $U$ then there exists an open neighborhood $\U\subset \diff^{1}_{m}(M)$ of $f$ such that the continuation $W_{g}$ of $W$ is minimal for a residual set of $g\in\U$. In particular, stable minimality is dense in $\U$.
\end{proposition}

\begin{proof}
 It follows from the hypotheses that $\phc^{W}(p)=M$. We also deduce from the hypotheses that there exists a continuation $W_{g}$ that is an expanding $g$-invariant foliation for all $g$ in a neighborhood $\U$ of $f$. It follows straightforwardly from \cite[Lemma 3.2]{NH19} that $\phc^{W_{g}}(p_{g})=M$, where $p_{g}$ is the continuation of $p$, in a certain neighborhood of $f$ which we continue to call $\U$ (see also Lemma \ref{lemma.phc.persists}). From \cite{BC} we know that generically in $\diff^{1}_{m}(M)$, the homoclinic class of every hyperbolic periodic point is the whole manifold. In particular, $\overline{W_{g}(p_{g})}=M$ for a residual set of $g\in\U$.  The minimality criterion (Proposition \ref{teo.minimality}) then implies that $W_{g}$ is minimal for a residual set of $g\in \U$. Theorem \ref{teo.stable.minimality} then implies that stable minimality is dense.
\end{proof}
 
Proposition \ref{proposition.criterion} does not require partial hyperbolicity.  This criterion can be applied, for instance, to the Ma\~n\'e example \cite{Ma1978}, which can easily be done volume preserving. This provides a non partially hyperbolic stably minimal example. Proposition above also imples the following
\begin{corollary}
Every diffeomorphism with a minimal expanding foliation $W$ such that $TW$ dominates an invariant sub-bundle and has a hyperbolic periodic point with unstable index equal to $\dim W$ can be approximated by diffeomorphisms having stably minimal foliations.  
\end{corollary}

\subsection{New examples} Using the Proposition \ref{proposition.criterion} above and its corollary we show in this subsection how to build new examples. Consider a volume preserving Anosov $f_{0}$ in the $3$-torus. Let us assume $T\T^{3}=E^{u}\oplus E^{s}$, where $Df|_{E^{u}}$ is expanding and $Df|_{E^{s}}$ is contracting, and $\dim E^{u}=1$.  Let $p$ be an $f_{0}$-fixed point and $q$ an $f_{0}$-periodic point. Let $f$ be a $C^{0}$-perturbation of $f_{0}$ supported in a very small neighborhood $U$ of $q$, not containing $p$, such that $f$ preserves the $W^{s}$-foliation and $f$ admits a dominated splitting of the form $T\T^{3}=E^{u}\oplus E^{cs}$. $f$ is not necessarily Anosov, nor partially hyperbolic. There is plenty of flexibility to obtain examples like this. See for instance \cite{NOH2020}. Observe that $p$ is also an $f$-fixed point. 
\subsubsection{Claim 1. $W^{s}_{f}(p)=W^{s}_{f_{0}}(p)$ } Consider a large disc $D\subset W^{s}_{f_{0}}(p)$ containing $p$ such that $D\cap U=\emptyset$. We want to show that $\bigcup_{n\geq0}f^{-n}(D)=W^{s}_{f_{0}}(p)$. Since $p$ is a hyperbolic fixed point, we know from the Stable manifold theorem that $W^{s}_{f}(p)=\bigcup_{n\geq 0}f^{-n}(D)$ is homeomorphic to a plane. We also know from Franks \cite{franks1970} that there is a semiconjugacy $h$ such that $d(h,id)$ can be made arbitrarily small by taking $U$ arbitrarily small, and $h\circ f=f_{0}\circ h$. Call $r_{n}(f)$, $r_{n}(f_{0})$ the internal radii of $f^{-n}(D)$, $f^{-n}_{0}(h(D))$, respectively. Since $f_{0}$ is Anosov and $h$ is close to the identity map, $r_{n}(f_{0})\to\infty$ as $n\to\infty$. Now, $h\circ f^{-n}(D)= f_{0}^{-n}\circ h(D)$. This implies that $|r_{n}(f)-r_{n}(f_{0})|\leq K<\infty$. Therefore, $r_{n}(f)\to\infty$ as $n\to\infty$, and hence we get the claim.
\subsubsection{Claim 2. $\phc_{f}^{u}(p)=M$} Since $f$ admits a dominated splitting of the form $T\T^{3}=E^{u}\oplus E^{cs}$, where $E^{u}$ is one-dimensional, $E^{u}$ is expanding. This implies there is a foliation transverse to the sub-bundle $E^{cs}$. In particular, since $W^{s}_{f}(p)=W^{s}_{f_{0}}(p)$, 
$W^{s}_{f}(p)$, $W^{s}_{f}(p)$ is dense, and it cuts transversely every leaf of the foliation $W^{u}$ tangent to $E^{u}$. Therefore $\phc_{f}^{u}(p)=M$.\par
Proposition \ref{proposition.criterion} above provides a $C^{1}$-dense set of diffeomorphisms $g$ in a neighborhood $\U$ of $f$ such that for each $g$, the continuation $W_{g}$ of $W^{u}$ is stably minimal. 
\subsection*{Acknowledgements:} We want to thank R. Ures for useful suggestions and discussions that helped improve this paper .

\bibliographystyle{alpha}
\bibliography{2018minimalitytesis}

\end{document}